\documentclass[dvipdfmx, 11pt, a4paper]{article}
\usepackage{amsmath}
\usepackage{amssymb}
\usepackage{amsthm}
\usepackage{amssymb}
\usepackage{verbatim}
\usepackage[dvipdfmx]{graphicx}

\begin{document}

\title{On a Combination of the Cyclic Nimhoff and Subtraction Games}

\author{Tomoaki Abuku\footnote{University of Tsukuba, Ibaraki, Japan} \and Masanori Fukui\footnote{Hiroshima University, Hiroshima, Japan}\footnote{Osaka Electro-Communication University, Osaka, Japan}
\and Ko Sakai\footnote{University of Tsukuba, Ibaraki, Japan}
\and Koki Suetsugu\footnote{Kyoto University, Kyoto, Japan}\footnote{Presently with National Institute of Informatics, Tokyo,
Japan}}

\theoremstyle{definition} 
\theoremstyle{plain}
\newtheorem{thm}{Theorem}[section]
\newtheorem{Lem}{Lemma}[section]
\newtheorem{proposition}{Proposition}[section]
\newtheorem{corollary}{Corollary}[section]

\theoremstyle{defn} 
\newtheorem{conjecture}{Conjecture}[section]
\newtheorem{defn}{Definiton}[section]
\newtheorem{exam}{Example}[section]
\newtheorem{rem}{Remark}[section]
\newtheorem{pict}{Figure}[section]
\newcommand{\Nat}{\mathbb{N}}


\date{}

\maketitle

\centerline{\bf Abstract}
\noindent
In this paper, we study a combination (called the generalized cyclic Nimhoff) of the cyclic Nimhoff and subtraction games. We give the $\mathcal{G}$-value of the game when all the $\mathcal{G}$-value sequence of subtraction games have a common $h$-stair structure.\\

\section{Introduction}
\label{intro}
\subsection{Impartial game}
This paper discusses ``impartial'' combinatorial games in normal rule, namely games with the following characters:
\begin{itemize}
\item\ Two players alternately make a move.
\item\ No chance elements (the possible moves in any given position is determined in advance).
\item\ The both players have complete knowledge of the game states.
\item\ The game terminates in finitely many moves.
\item\ The both players have the same set of the possible moves in any position. (impartial)
\item\ The player who makes the last move wins. (normal)
\end{itemize}

Throughout this paper, we suppose that all game positions are
``short'', namely there are limitedly many positions that can be reached from the initial position, and any position cannot appear twice in a play.

\begin{defn}[outcome classes]
A game position is called an $\mathcal{N}$-position (resp. a $\mathcal{P}$-position) if the first player (resp. the second player) has a winning strategy.
\end{defn}

Clearly, all impartial game positions are classified into $\mathcal{N}$-positions or $\mathcal{P}$-positions.

If $G$ is an $\mathcal{N}$-position, there exists a move from $G$ to a $\mathcal{P}$-position.
If $G$ is a $\mathcal{P}$-position, there exists no move from $G$ to a $\mathcal{P}$-position.

\subsection{Nim and $\mathcal{G}$-value}
Nim is a well-known impartial game with the following rules:
\begin{itemize}
\item\ It is played with several heaps of tokens.
\item\ The legal move is to remove any number of tokens (but at least one token) from any single heap.
\item\ The end position is the state of no heaps of tokens.
\end{itemize}

We denote by $\mathbb{N}_{0}$ the set of all nonnegative integers. 
\begin{defn}[nim-sum]
The value obtained by adding numbers in binary form without carry is called nim-sum. The nim-sum of nonnegative integers $m_{1},\ldots,m_{n}$  is written by 
\begin{align*}
m_{1}\oplus \cdots \oplus m_{n}.
\end{align*}
\end{defn}

The set $\mathbb{N}_{0}$ is isomorphic to the direct sum of countably many $\mathbb{Z}/2\mathbb{Z}$'s.

\begin{defn}[minimum excluded number]
Let $T$ be a proper subset of $\mathbb{N}_{0}$. Then $\mathrm{mex}\ $$T$ is defined to be the least nonnegative integer not contained in $T$, namely
\begin{align*}
\mathrm{mex}\ T=\mathrm{min} (\mathbb{N}_0 \setminus T).
\end{align*}
\end{defn}

\begin{defn}[$\mathcal{G}$-value]
Let $G$ and $G'$ be game positions. The notation $G \rightarrow G'$ means that $G'$ can be reached from $G$ by a single move.  
The value $\mathcal{G}(G)$ called the $\mathcal{G}$-value (or nim value or Grundy value or SG-value, depending on authors) of $G$ is defined as follows:
\begin{align*}
  \mathcal{G}(G)=\mathrm{mex} \{\mathcal{G}(G')\mid G \rightarrow G'\}. 
\end{align*}
\end{defn}

The following theorem is well-known.

\begin{thm}[\cite{Grundy}, \cite{Sprague}]
$\mathcal{G}(G)=0$ if and only if $G$ is a $\mathcal{P}$-position.
\end{thm}

Therefore, we only need to decide the $\mathcal{G}$-value of positions for winning strategy in impartial games and the $\mathcal{G}$-value is also useful for analysis of the disjunctive sum of games. If $G$ and $H$ are any positions of (possibly different) impartial games, the disjunctive sum of $G$ and $H$ (written as $G + H$) is defined as follows: each player must make a move in either $G$ or $H$ (but not both) on his turn. 
\begin{thm}[\cite{Grundy}, \cite{Sprague}]
Let $G$ and $H$ be two game positions. Then
\begin{align*}
\mathcal{G}(G+H)=\mathcal{G}(G)\oplus \mathcal{G}(H).
\end{align*}
\end{thm}

\begin{defn}[periodic]
Let $a=(a_0, a_1, a_2, \ldots)$ be a sequence of integers. We say that $a$ is periodic with period $p$ and preperiod $n_0$, if we have 
\begin{center}
$a_{n+p}=a_n$ for all $n\geq n_0$.
\end{center}
We say that $a$ is purely periodic if it is periodic with preperiod $0$.
\end{defn}

\begin{defn}[arithmetic periodic]
Let $a=(a_0, a_1, a_2, \ldots)$ be a sequence of integers. We say that $a$ is arithmetic periodic with period $p$, preperiod $n_0$, and saltus $s$, if we have 
\begin{center}
$a_{n+p}=a_n+s$ for all $n\geq n_0$.
\end{center}
\end{defn}

\begin{defn}[$\mathcal{G}$-value sequence of a heap game]
Assume $H$ is a heap game and let $\mathcal{G}_{H}(m)$ be the $\mathcal{G}$-value of a single heap with $m$ tokens. The we call sequence
\begin{align*}
\mathcal{G}_{H}(0), \mathcal{G}_{H}(1), \ldots
\end{align*}
the $\mathcal{G}$-value sequence of $H$.
\end{defn}

Shortly after Bouton published studies on Nim \cite{Bouton}, Wythoff conducted research on $\mathcal{P}$-position
of a game which is nowadays called Wythoff's Nim \cite{Wythoff}. 
Wythoff's Nim is a well-known impartial game with the following rules:
\begin{itemize}
\item\ The legal move is to remove any number of tokens from a single heap (as in Nim) or remove the same number of tokens from both heaps.
\item\ The end position is the state of no heaps of tokens.
\end{itemize}

Wythoff's work was one
of the earliest researches on heap games 
which permit the players to remove tokens from
more then one heap at the same time. Another early research was by
Moore \cite{Moore}. Let $k$ be a fixed given number. In Moore's game the player can remove tokens from
less then $k$ heaps at the same time without any restriction.

\subsection{The Cyclic Nimhoff}
Nimhoff was extensively
researched by Fraenkel and Lorberbom \cite{Fraenkel}. 
Let $R$ be a subset of $\mathbb{N}_0^n$ not
containing $(0,\dots,0)$. A position of Nimhoff is $m$-heaps of tokens. The
moves are of two types: each player can remove any positive number of
tokens from a single heap, or remove $s_i$ tokens from
the $i$th heap for $i=1,\dots,n$ such that $(s_1,\ldots,s_n)\in R$.
In particular they researched the cyclic Nimhoff
in the case that $R=\{(s_1,\ldots,s_n) \mid 0<\displaystyle \sum_{i=1}^{n} s_i<h\}$, where
$h$ is a fixed positive integer \cite{Fraenkel}. The $\mathcal{G}$-value of position
$(x_1,x_2,\ldots,x_n)$ in the cyclic Nimhoff is
$$
\mathcal{G}(x_1,x_2,\ldots,x_n) = \left(\displaystyle \bigoplus_i^{} \left\lfloor \frac{x_i}{h}\right\rfloor\right)h + \left(\displaystyle \sum_i^{} x_i \right)\bmod h,
$$
where $\displaystyle \bigoplus_i^{}{a_i}$ denotes the nim-sum of all $a_i$'s.

\subsection{The Subtraction Games}
Let $S$ be a set of positive integers.
In the subtraction game $\mathrm{Subtraction}(S)$, the only legal moves
are to remove $s$ tokens from a heap for some $s\in S$.
In particular, Nim is $\mathrm{Subtraction}(\Nat_+)$, where $\mathbb{N}_{+}$ is the set of all positive integers.
There are a lot of preceding studies on subtraction games \cite{Berlekamp}. 
For example，All-but subtraction games All-but$(S)$
(i.e. $\mathrm{Subtraction}(\Nat_+\setminus S)$ such that $S$ is a 
finite set) were studied in detail by 
Angela Siegel \cite{ASiegel}. She proved that 
the $\mathcal{G}$-value sequence
is arithmetic periodic and characterized some cases in which the sequence is purely periodic. 

\section{The Generalized Cyclic Nimhoff}

We define the generalized cyclic Nimhoff as a combination of
the cyclic Nimhoff and subtraction games as follows.

\begin{defn}[generalized cyclic Nimhoff]

Let $h$ be a fixed positive integer and $S_1, S_2, \ldots
,S_n$ are sets of positive integers. Let $(x_1,x_2,\ldots,x_n)$ be an
ordered $n$-tuple of non-negative integers. We define subsets $X_1$, $X_2$,
$\cdots$, $X_n$, $Y$ of the set of $n$-tuples of non-negative
integers as follows:
\begin{align*}
X_1&=\{(x_1-s_1,x_2,\ldots,x_n)\mid s_1 \in
S_1\}\\ X_2&=\{(x_1,x_2-s_2,\ldots,x_n)\mid s_2 \in
S_2\}\\ &\vdots\\ X_n&=\{(x_1,x_2,\ldots,x_n-s_n)\mid s_n \in
S_n\}\\ Y&=\{(x_1-s_1,x_2-s_2,\ldots,x_n-s_n)\mid 0<\displaystyle \sum_{i=1}^{n} s_i< h\}.
\end{align*}
In the generalized cyclic Nimhoff GCN$(h; S_1,S_2,\ldots,S_n)$, the set of legal moves from
position $(x_1,x_2,\ldots,x_n)$ is $X_1\cup X_2\cup\cdots\cup X_n\cup
Y$.
\end{defn}

\begin{defn}
Let $a=\{a(x)\}^{\infty}_{x=0}$ be an arbitrary sequence of
non-negative integers. The $h$-stair $b=\{b(x)\}^{\infty}_{x=0}$
of $a$ is defined by the following:
\begin{align*}
b(xh+r)=a(x)h+r
\end{align*}
for all $x\in \mathbb{N}$ and for all $r=0,1,\cdots,h-1$.
\end{defn}

\begin{exam}
If $a=0, 0, 1, 5, 4, \ldots$, then the $3$-stair of $a$ is\\ $b=0,1,2,0,1,2,3,4,5,15,16,17,12,13,14,\ldots$.
\end{exam}

Let us denote the $\mathcal{G}$-value sequence of 
$\mathrm{Subtraction}(S)$ by $\{G_{S}(x)\}_{x=0}^{\infty}$.

\begin{thm}\label{generalizednimhoff}
Let $a_1,a_2, \ldots, a_n$ be arbitrary sequences of non-negative
integers. Let $(x_1,x_2, ... , x_n)$ be a game position of the generalized cyclic Nimhoff GCN$(h; S_1,S_2,\ldots,S_n)$.
If $\{G_{S_i}(x)\}$ is the $h$-stair of sequence $a_i$ for all
$i$ $(1\leq i \leq n)$, then
\begin{align*}
\mathcal{G}(x_1,x_2,\ldots,x_n) 
= \left(\displaystyle \bigoplus_i^{}\left\lfloor \frac{G_{S_i}(x_i)}{h}
\right\rfloor\right)h+\left(\displaystyle \sum_i^{} x_i \right)\bmod h.
\end{align*}
\end{thm}

\begin{proof}

For each $i=1,\dots,n$, let $x_i =  q_i h + r_i$ where $0 \leq r_i < h$.
Since
$G_{S_i}$ is the $h$-stair of sequence $a_i$, $G_{S_i}(x_i)$ =
$a_i(q_i)h + r_i$.  
In other words, note that 
$$
\left\lfloor \dfrac{x_i}{h}\right\rfloor = q_i,\quad 
\left\lfloor \dfrac{G_{S_i}(x_i)}{h}\right\rfloor = a_i(q_i),\quad
x_i\equiv G_{S_i}(x_i)\equiv r_i\pmod h.
$$

The proof is by induction on $(x_1,x_2, \ldots, x_n)$.

Let
\begin{align*}
Q(x_1,x_2, \ldots, x_n) &=
\displaystyle \bigoplus_i^{}\left\lfloor \frac{G_{S_i}(x_i)}{h} \right\rfloor
= \displaystyle \bigoplus_i^{}{a_i}(q_i),\\
R(x_1,x_2,\ldots,x_n)&=\displaystyle \sum_i^{} x_i \bmod h =
\displaystyle \sum_i^{} G_{S_i}(x_i) \bmod h =
\displaystyle \sum_i^{} r_i \bmod h.
\end{align*}
Then, it is sufficient to prove that
\begin{align*}
\mathcal{G}(x_1,x_2,\ldots,x_n)= Q(x_1,x_2, \ldots, x_n)h+R(x_1,x_2,\ldots,x_n).
\end{align*}

First, we show that for any $k < Q(x_1,x_2, \ldots, x_n)h
+R(x_1,x_2,\ldots,x_n)$, there exists a position $(x'_1,x'_2, \ldots, x'_n) \in
X_1\cup X_2\cup\cdots\cup X_n\cup Y$ such that
$\mathcal{G}(x'_1,x'_2,\ldots,x'_n) =k$. There are two cases.

Case that $Q(x_1,x_2, \ldots, x_n)h \leq
k < Q(x_1,x_2, \ldots, x_n)h+R(x_1,x_2,\ldots,x_n)$:

In this case, $k$ can be written in form $Q(x_1,x_2, \ldots, x_n)h +
k'$ by $k'$ such that $0 \leq k' < R(x_1,x_2,\ldots,x_n)$.  Since $0 <
R(x_1,x_2,\ldots,x_n) - k' \leq R(x_1,x_2,\ldots,x_n) = 
\sum_i^{} r_i \bmod h$ and $0 < R(x_1,x_2,\ldots,x_n) - k' < h$,
there exist $(k_1,k_2,\ldots, k_n)$ such that $k_1 + k_2 + \cdots +
k_n= R(x_1,x_2,\ldots,x_n) - k'$ and $k_j \leq r_j$ for each $j$. Then
$(x_1 - k_1, x_2-k_2, \ldots x_n-k_n) \in Y$.  In addition, $Q(x_1-k_1,x_2-k_2,\ldots,x_n-k_n) = Q(x_1,x_2,\ldots,x_n)$ and $R(x_1-k_1,x_2-k_2,\ldots, x_n-k_n) = R(x_1,x_2,\ldots,x_n) - (k_1 + k_2 + \cdots + k_n) = k'$. Therefore,
$\mathcal{G}(x_1-k_1,x_2-k_2,\ldots,x_n-k_n) = Q(x_1,x_2,\ldots,x_n)h
+ k' = k$ from induction hypothesis.

Case that $k<Q(x_1,x_2,\ldots,x_n)h$:

In this case, $k$ can be written in form $Q'h+k'$ by $Q'$ and $k'$
such that $Q' < Q(x_1,x_2,\ldots,x_n) =
\bigoplus_i^{}{a_i}(q_i)$ and $0 \leq k' <h$.

According to the nature of nim-sum, there exists $j$ and $g$ which
satisfy $Q' = a_1(q_1) \oplus a_2(q_2) \oplus \cdots \oplus
a_{j-1}(q_{j-1}) \oplus g \oplus a_{j+1}(q_{j+1})\oplus \cdots
\oplus a_n(q_n)$ and $g < a_j(q_j)$.  Without loss of generality,
we assume $j=1$. That is, there exist $g<{a_1}(q_1)$ which satisfies
$Q' = g \oplus {a_2}(q_2) \oplus \cdots \oplus {a_n}(q_n) $.  
If we put $r'_1$ to satisfy that $(r'_1 + r_2 + r_3
+ \cdots + r_n) \bmod h = k'$ and $0 \leq r'_1 < h$,
then $gh+r_1'<a_1(q_1)h\leq a_1(q_1)h+r_1=G_{S_1}(x_1)$, and therefore, 
there exists $x'_1$ such that $G_{S_1}(x_1')=gh+r_1'$ and $x_1-x'_1 \in S_1$. 
Thus, we have $(x'_1,x_2,\ldots, x_n) \in X_1$. Therefore,
\begin{align*}
\mathcal{G}(x'_1,x_2,\ldots,x_n) &=\left(\left\lfloor\dfrac{G_{S_1}(x_1')}{h} \right\rfloor
\oplus\left\lfloor\dfrac{G_{S_2}(x_2)}{h} \right\rfloor\oplus \cdots
\oplus\left\lfloor\dfrac{G_{S_n}(x_n)}{h} \right\rfloor\right)h\\
  &\quad +(x'_1 + x_2 + \cdots + x_n)\bmod h\\
&=(g \oplus {a_{2}}(q_{2})\oplus \cdots \oplus {a_n}(q_n))h+k'=Q'h+k'= k
\end{align*}
from induction hypothesis.

Next, we show that, if 
$(x_1,x_2,\ldots,x_n)$ $\rightarrow$ $(x'_1,x'_2,\ldots,x'_n)$,
then  
$$
Q(x_1,x_2,\ldots,x_n)h + R(x_1,x_2,\ldots,x_n)\ne
Q(x'_1,x'_2,\ldots,x'_n)h + R(x'_1,x'_2,\ldots,x'_n).
$$ 
Cleary, the claim is true if $(x'_1,x'_2,\ldots,x'_n)$ is in $Y$,
since $R(x'_1,x'_2,\ldots,x'_n) \neq R(x_1,x_2,\ldots,x_n)$.
Therefore, we assume that
$(x'_1,x'_2,\ldots,x'_n)$ is in $X_1$ without loss of generality, 
namely $x'_j= x_j$ $(j>1)$ and
$x_1-x'_1 \in S_{1}$. Let $x'_1 = q'_1 h + r'_1\ (0\leq r'_1<h)$. If
$Q(x_1,x_2,\ldots,x_n)h + R(x_1,x_2,\ldots,x_n) =
Q(x'_1,x_2,\ldots,x_n)h + R(x'_1,x_2,\ldots,x_n)$, then we have
$Q(x_1,x_2,\ldots,x_n) = Q(x'_1,x_2,\ldots,x_n)$ and
$R(x_1,x_2,\ldots,x_n) = R(x'_1,x_2,\ldots,x_n)$. Then, $r_1 = r'_1$
since $R(x_1,x_2,\ldots,x_n) = R(x'_1,x_2,\ldots,x_n)$, and
$\lfloor{G_{S_1}(x_1)}/{h}\rfloor =\lfloor{G_{S_1}(x_1')}/{h}
\rfloor$ since $Q(x_1,x_2,\ldots,x_n) =
Q(x'_1,x_2,\ldots,x_n)$. Therefore $G_{S_1}(x_1)=G_{S_1}(x_1')$, but it is
impossible because $x_1-x'_1 \in S_{1}$.
\end{proof}

There are a variety of subtraction games with the 
$h$-stair of a simple integer sequence
as their $\mathcal{G}$-value sequence.

\begin{exam}[Nim]

For any $h$, $G_{\mathbb{N}_{+}}(x) = x = \left(\left\lfloor \frac{x}{h} \right\rfloor\right)h+(x \bmod h)$.
\end{exam}

\begin{exam}[Subtraction($\{1,\ldots,l-1\}$) and its variants]

If $\{1,\ldots,l-1\} \subset S \subset \mathbb{N}_+\setminus\{kl\mid l\in \mathbb{N}_+\}$ and $h\mid l$, then
\begin{align*}
G_{S}(x) = x\bmod l = \left(\left\lfloor \frac{x \bmod l}{h}
\right\rfloor\right)h+((x \bmod l)\bmod h).
\end{align*}
\end{exam}

\begin{exam}[All-but($\{h,2h,\ldots,kh\}$)]

If $S=\mathbb{N}_{+}\setminus \{h,2h,\ldots,kh\}$, then $G_{S}(x)$ is the $h$-stair of 

$\{\underbrace{0,0,\ldots,0}_{k+1}, \underbrace{1,1,\ldots,1}_{k+1},\underbrace{2,2,\ldots,2}_{k+1}\ldots\}$.
\end{exam}

\begin{exam}[All-but($\{s_1,s_2\}$) \cite{ASiegel}]

If $s_2>s_1$, then $G_{S}(x)$ is the $s_1$-stair of a sequence of positive integers.
\end{exam}

Theorem \ref{generalizednimhoff} allows us to combine several subtraction games which have $\mathcal{G}$-value sequences of form $h$-stair for common $h$.
For example, for GCN($4$; $\mathbb{N}_{+}$, $\{1,2,3,4,5,6,7\}$, $\mathbb{N}_{+}\setminus \{4,8\}$),
we have the following:
\begin{align*}
\mathcal{G}(x_1,x_2,x_3) 
=\left(\left\lfloor \frac{x_1}{4}\right\rfloor\oplus\left\lfloor \frac{x_2 \bmod 8}{4}\right\rfloor\oplus\left\lfloor \frac{x_3}{12}
\right\rfloor\right)\times4+(x_1+x_2+x_3) \bmod 4.
\end{align*}

Suppose that a subtraction set $S$ is given. Then we can define a new subtraction set $S'$ such that the $\mathcal{G}$-value sequence of $\mathrm{Subtraction}(S')$ is the $h$-stair of the $\mathcal{G}$-value sequence of $\mathrm{Subtraction}(S)$.

\begin{thm}
Let $S$ be an arbitrary subtraction set and let $S'=\mathbb{N}_{+}\setminus \{(\mathbb{N}_{+}\setminus S)h\}$. Then 
\begin{align*}
G_{S'}(n) =G_{S}\left(\left\lfloor \frac{n}{h} \right\rfloor\right)h+(n \bmod h).
\end{align*}
\end{thm}

\begin{proof}
Let $n=qh+i$ where $0\leq i<h$. Then the formula to be shown is $G_{S'}(qh+i)=G_{S}(q)h+i$. The proof is by induction on $n$ $(=qh+i)$.\\

First, we show that there exists a move to a position with any smaller $\mathcal{G}$-value $rh+k$ than $G_{S}(q)h+i$.
There are two cases.\\

Case that $r=G_{S}(q)$ and $k < i$:\\
Since $0< i-k <h$, there exists a move $gh+i \rightarrow gh+k$ and we have
\begin{align*}
G_{S'}(gh+k)=G_{S}(q)h+k=rh+k
\end{align*}
by induction hypothesis.\\
Case that $r<G_{S}(q)$ and $0\leq k<h$:\\

By the definition of $G_{S}(q)$, there exists $q'$ such that $q-q'\in S$, $G_{S}(q')=r$ and 
\begin{align*}
G_{S'}(q'h+k)=G_{S}(q')h+k=rh+k
\end{align*}
by induction hypothesis.
So we only need to prove that there is move to $q'h+k$.\\
If $i \neq k$, clearly there exists a move $qh+i \rightarrow q'h+k$.\\

Assume that $i = k$ and that there does not exist a move $qh+i \rightarrow q'h+k$. Then 
\begin{align*}
(q-q')h \notin S'\Rightarrow (q-q')h \in (\mathbb{N}_{+}\setminus S)h\Rightarrow q-q' \in (\mathbb{N}_{+} \setminus S)\Rightarrow q-q' \notin S, 
\end{align*}
which is a contradiction.\\
Next, we show that, if $n=qh+i \rightarrow n'=q'h+k$, then 
\begin{center}
$G_{S'}(n')$ $\neq$ $G_{S}(q)h+i$.
\end{center}
If $G_{S'}(n')=G_{S}(q)h+i$, then we have $G_{S}(q)=G_{S}(q')$ and $k=i$ by induction hypothesis, but it is impossible by the definition of $G_{S}(q)$. Because
\begin{align*}
(q-q') \notin S\Rightarrow (q-q') \in (\mathbb{N}_{+}\setminus S)\Rightarrow (q-q')h \in (\mathbb{N}_{+} \setminus S)h\\
\Rightarrow (q-q')h \notin \mathbb{N}_{+}\setminus\{(\mathbb{N}_{+} \setminus S)h\}\Rightarrow n-n'\notin S'.
\end{align*}
\end{proof}

\addcontentsline{toc}{section}{{\bf References}}

The affiliations of the authors.\\
Tomoaki Abuku: University of Tsukuba, Ibaraki, Japan.\\
Masanori Fukui: Hiroshima University, Hiroshima, Japan and Osaka Electro-Communication University, Osaka, Japan.\\
Ko Sakai: University of Tsukuba, Ibaraki, Japan.\\
Koki Suetsugu: Kyoto University, Kyoto, Japan. Presently with National Institute of Informatics, Tokyo, Japan.
\end{document}